\documentclass[10pt]{article}

\pdfoutput = 0

\usepackage[a4paper]{geometry}
\usepackage[utf8]{inputenc}
\usepackage[T1]{fontenc}
\usepackage[english]{babel}
\usepackage[babel=true,kerning=true]{microtype}
\usepackage{lmodern}

\usepackage{pifont}
\rmfamily
\DeclareFontShape{T1}{lmr}{b}{sc}{<->ssub*cmr/bx/sc}{}
\DeclareFontShape{T1}{lmr}{bx}{sc}{<->ssub*cmr/bx/sc}{}

\usepackage{tikz}
\usepackage{wrapfig} 
\usepackage[all]{xy}

\usetikzlibrary{shapes,patterns,plotmarks}
\usetikzlibrary{arrows}
\usepackage{amsfonts,amssymb}
\usepackage{amsmath,amscd}

\usepackage{array}\usepackage{multirow}
\usepackage{epsfig,graphics,graphicx,xcolor}
\usepackage{stmaryrd}

\usepackage[babel=true]{csquotes}
\usepackage{enumitem}
\usepackage{mathtools}
\usepackage[numbers]{natbib}

\usepackage{color}

\definecolor{lightred}{rgb}{1, 0.5, 0.5}
\definecolor{lightgrey}{rgb}{0.6, 1, 0.6}
\usepackage{amsthm}

\DeclareFontFamily{U}{mathx}{\hyphenchar\font45}
\DeclareFontShape{U}{mathx}{m}{n}{
      <5> <6> <7> <8> <9> <10>
      <10.95> <12> <14.4> <17.28> <20.74> <24.88>
      mathx10
      }{}
\DeclareSymbolFont{mathx}{U}{mathx}{m}{n}
\DeclareMathSymbol{\bigtimes}{1}{mathx}{"91}

    \makeatletter
    \newcounter {subsubsubsection}[subsubsection]
    \renewcommand\thesubsubsubsection{\thesubsubsection .\@arabic\c@subsubsubsection}
    \newcommand\subsubsubsection{\@startsection{subsubsubsection}{4}{\z@}%
                                         {-3.25ex\@plus -1ex \@minus -.2ex}%
                                         {1.5ex \@plus .2ex}%
                                         {\normalfont\normalsize\itshape}}
    \newcommand*\l@subsubsubsection{\@dottedtocline{4}{10.0em}{4.1em}}
    \newcommand*{\subsubsubsectionmark}[1]{}
    \makeatother

  \theoremstyle{plain}
\newtheorem{theorem}{Theorem}[section]

\newtheorem{conjecture}[theorem]{Conjecture} 
\newtheorem{lemma}[theorem]{Lemma} 
 
\newtheorem{example}[theorem]{Example} 
 
  \theoremstyle{remark}
\newtheorem{remark}[theorem]{Remark}
  \theoremstyle{definition}
\newtheorem{definition}[theorem]{Definition}

\newcommand{\G}{\mathcal{G}}

\newcommand{\NN}{\mathbb{N}}
\newcommand{\ZZ}{\mathbb{Z}}
\newcommand{\RR}{\mathbb{R}}
\newcommand{\D}{\mathbb{D}}
\renewcommand{\SS}{\mathbb{S}}

\renewcommand{\i}{\iota}

\newcommand{\<}{\left\langle\!\left\langle}
\renewcommand{\>}{\left\rangle\!\left\rangle}

\newcommand{\I}{\operatorname{I}}
 
\newcommand{\rot}{\operatorname{rot}} 
\newcommand{\FH}{\operatorname{FH}}
\newcommand{\T}{\operatorname{T}}

\newcommand{\p}{^\prime}

\newcommand{\pp}{^{\prime\prime}}

\newcommand{\lk}{\operatorname{lk}}
\newcommand{\far}{\operatorname{far}}
\newcommand{\Ker}{\operatorname{Ker}}
\renewcommand{\tt}{$3$-$3$}
\newcommand{\td}{$3$-$2$}
\newcommand{\du}{$2$-$1$}
\newcommand{\dd}{$2$-$2$}
\newcommand{\E}{\operatorname{E}}

\newcommand{\K}{\mathcal{K}}

\renewcommand{\<}{\left\langle}
\renewcommand{\>}{\right\rangle}

\title{Combinatorial cohomology of the space of long knots}
\author{Arnaud Mortier}
\date{\today}

\begin{document}

\maketitle

\begin{abstract}
\footnotesize
The motivation of this work is to define cohomology classes in the space of knots that are both easy to find and to evaluate, by reducing the problem to simple linear algebra.
We achieve this goal by defining a combinatorial graded cochain complex, such that the elements of an explicit submodule in the cohomology define algebraic intersections with some \enquote{geometrically simple} strata in the space of knots. Such strata are endowed with explicit co-orientations, that are canonical in some sense. The combinatorial tools involved are natural generalisations (degeneracies) of usual methods using arrow diagrams.
\end{abstract}

\tableofcontents
\vspace*{0.5cm}
The paper is organised as follows.

In Section~\ref{sec:leaf}, we build a prototypical cochain complex which contains all the essential combinatorics while using the most simple input, namely a finite collection of finite subsets of $\RR$ (a \textit{colored leaf diagram}). The point of this preliminary is both theoretical and to point out clearly that this part of our construction does not depend on the material introduced further.

In Section~\ref{sec:strata}, we show that the incidence signs of the previous cochain complex are of topological nature, as they are an essential ingredient in the computation of the boundary of the meridian discs of some \enquote{geometrically simple} strata in the space of knots, provided that these discs are correctly oriented. This property canonically defines a co-orientation of simple strata.

Simple strata are represented by means of degenerated Gauss diagrams, i.e., whose arrows are allowed to meet on the base circle. Then, in Section~\ref{sec:main},  similarly to Polyak-Viro's formulas for finite-type invariants, we define cochains by counting subconfigurations in those diagrams, with weights given by products of writhes. A little twist appears here: we do not count the signs of arrows that participate in singularities; these signs contribute implicitly, via the definition of the canonical co-orientation. 

At the end of the section we construct the main cochain complex, which is a slightly twisted version of that of Section~\ref{sec:leaf}, and construct a Stokes formula relating it with the boundary maps from Section~\ref{sec:strata},  that model the meridians of simple strata. The announced result follows.

The last section is a review of examples, including new formulas for the low degree Vassiliev invariants obtained by integrating $1$- and $2$-cocycles over some canonical $1$- and $2$-chains. In particular we give a method for integrating our $1$-cocycle formulas into knot invariants without any computations, over the two main canonical cycles in the space of knots -- namely the Gramain loop, and the Fox-Hatcher loop.

\subsection*{Acknowledgements}

I wish to express my full gratitude to Seiichi Kamada for his hospitality and the opportunity he gave me to pursue research for one year in excellent conditions in Osaka City University. I am grateful to Allen Hatcher, Thomas Fiedler, Victoria Lebed and Ryan Budney for helpful discussions and comments.

\section{Cohomology of coloured leaf diagrams in $\RR$}\label{sec:leaf}


\subsection{Polygons}

A \textit{polygon} is a finite subset of the oriented based circle $\SS^1=\RR\cup \left\lbrace \infty \right\rbrace$. We make no distinction between a polygon and the corresponding singular $0$-chain in $C_0(\SS^1,\ZZ_2)$. It is said to be \textit{even} or \textit{odd} according to the parity of its cardinality -- in other words, odd polygons are those representing the non-trivial homology class in $H_0(\SS^1,\ZZ_2)$.

Let $P$ and $P\p$ be two disjoint even polygons. Then they have a well-defined \textit{linking number}, denoted by $\lk(P,P\p)\in \ZZ_2=\left\lbrace \pm 1 \right\rbrace $, which is the algebraic intersection between $P\p$ and any $1$-chain in $C_1(\SS^1,\ZZ_2)$ whose boundary is $P$. The map $\lk$ is symmetric, and bilinear in the sense that if $P$ and $P\p$ are disjoint, as well as $P$ and $P\pp$, then $$\lk(P,P\p+P\pp)=\lk(P,P\p)+ \lk(P,P\pp).$$
If $P\p$ is odd and $P$ has two elements, again with $P\cap P\p=\emptyset$, we extend the notation by setting $$\lk(P,P\p)= \lk(P\p,P)\stackrel{\text{def}}{=}(-1)^{ \sharp \, \left[ \min(P),\,\max(P)\right]\, \cap \,P\p}, $$
where we agree that the point $\infty$ is \textit{greater} than any real number.
Note that the same formula holds when $P\p$ is even. The map $\lk$ can then be extended by symmetry and bilinearity to any couple of disjoint polygons at least one of which is even.

We define a partial order on the set of polygons by setting:
$$P< P\p \Longleftrightarrow \inf(P)< \inf(P\p).$$

\subsection{coloured leaf diagrams}

A \textit{(coloured) leaf diagram} in $\RR= \SS^1\setminus \left\lbrace\infty \right\rbrace $ is a finite collection of pairwise disjoint polygons, none of which contains $\infty$. The elements of the polygons are called \textit{leaves} of the diagram, and two leaves from the same polygon are said to have the same \textit{colour}. The terminology is inspired from the fact that such diagrams are meant to be later completed into \textit{tree diagrams} by connecting all leaves of a same colour by an abstract tree. We define two $\ZZ$-valued complexities associated with a leaf diagram $L$:
\begin{itemize}
\item The \textit{Gauss degree} $\deg(L)$, which is the total number of leaves minus the number of colours (polygons) in $L$.
\item The \textit{codimension} $\i(L)$, or \textit{cohomological degree}, which is the total number of leaves minus \textit{twice} the number of colours in $L$.\end{itemize}
The term \enquote{Gauss degree} comes from the theory of chord diagrams, where it denotes the number of chords. For instance, a leaf diagram with $d$ polygons, all of which have cardinality $2$, has Gauss degree $d$ and
codimension $0$.

Leaf diagrams are regarded up to orientation preserving homeomorphisms of the real line $\SS^1\setminus \left\lbrace \infty\right\rbrace $. The $\ZZ$-module freely generated by equivalence classes of leaf diagrams of degree $d$ and codimension $i$ is denoted by $\mathcal{L}_d^i$. Note that $\mathcal{L}_d^i$ is always finitely generated, and is trivial whenever $i$ is greater than $d-1$.

\begin{remark}\label{rem:isolated}
Special attention should be paid to polygons with only one leaf. Such a polygon contributes $-1$ to the codimension, and has no effect on the Gauss degree. They are actually the only reason why the cohomological degree is not bounded and $\NN$-valued. In our main application for this theory, such polygons are naturally excluded, and the spaces of diagrams with fixed Gauss degree are finitely generated. However, it is harmless to allow them in the prototypical cochain complex, and there may be a theoretical interest to study their meaning and the relations between the main and \enquote{reduced} cohomology theories. 
\end{remark}

\subsection{The $\varepsilon$ signs and prototypical complex}

Let $L$ be a leaf diagram. An \textit{edge} of $L$ is a closed connected part of the circle that lies between two neighboring leaves of $L$ -- in particular, an edge cannot contain a leaf in its interior, and it cannot contain $\infty$. An edge is called \textit{admissible} if its two boundary points have different colours. From such an edge, we construct a new leaf diagram $L_e$ in the following way: the polygons of $L_e$ are the polygons of $L$, except for the two that have a leaf at the boundary of $e$: those two are merged into a single polygon in $L_e$, and one of the two boundary points of $e$ is removed from it (which one exactly has no effect on the resulting diagram up to positive homeomorphism of $\RR$).

One easily checks the following relations:
\begin{eqnarray*}
\deg(L_e)&=&\deg(L),\hspace*{1cm} \text{and}\\
\i(L_e)&=&\i(L)+1.
\end{eqnarray*}

Consider the linear maps $\mathcal{L}_d^{i-1} \rightarrow \mathcal{L}_d^i$ defined on each generator by the formula \begin{eqnarray*}
\,\,\,\, L &\mapsto & \sum_{ e \text{ admissible}} L_e. \end{eqnarray*}

It is easy to see that using $\ZZ_2$ coefficients, these maps turn the collection of spaces $\mathcal{L}_d^i$ into a graded cochain complex. Our goal is to define signs to lift this complex over $\ZZ$.

\subsubsection*{The global sign}

Let $P$ be an odd polygon in a leaf diagram $L$. We define the \textit{odd index} of $P$ as the parity of the number of odd polygons in $L$ that are greater than $P$. Using the convention that a boolean expression has value $-1$ when it is true and $1$ otherwise, this can be written:
$$\text{Odd}(P,L)\stackrel{\text{def}}{=} \prod_{P\p\text{ odd}}\left(P<P\p \right).$$
We extend this definition to all polygons by setting $\text{Odd}(P)=1$ whenever $P$ is even.

Let $e$ be an admissible edge in $L$, bounded by the leaves $v$ and $w$ lying respectively in the polygons $P_v$ and $P_w$. Also, denote by $P_{vw}$ the polygon of $L_e$ that results from the merging of $P_v$ and $P_w$.

We define the \textit{global sign} associated with the edge $e$ in $L$ by 
$$\sigma_{\text{glo}}(e,L) \stackrel{\text{def}}{=} \text{Odd}(P_v,L)\cdot\text{Odd}(P_w,L)\cdot\text{Odd}(P_{vw},L_e).$$
This will be the only contribution to the signs in the coboundary maps that depends on polygons located far from $e$. 

\begin{remark}\label{rem:minus} When $P_v$ and $P_w$ are both odd, both booleans $(P_v<P_w)$ and $(P_w<P_v)$ appear in $\sigma_{\text{glo}}$, which results in a minus sign. \end{remark}

\subsubsection*{The local sign}

From now on, for the sake of lightness, we omit to mention that every sign depends on $L$, since other diagrams like $L_e$ will not contribute any more.

If $x$ is a leaf in $L$, we denote by $P_x$ the polygon that contains it. Define the \textit{evenisation} of $P_x$ with respect to $x$ as $$P_x^{(x)}= \begin{array}{ll}
 P_x+x &\text{ if $P_x$ is odd,}\\
P_x & \text{ if $P_x$ is even.}\end{array}
$$ As a set, $P_x+x$ corresponds to $P_x\setminus \left\lbrace x\right\rbrace$, so that the polygon $P_x^{(x)}$ is always even. 

As previously, let $e$ be an admissible edge in $L$, bounded by the leaves $v$ and $w$. Recall the convention that a boolean expression takes value $-1$ when it is true and $1$ otherwise. 

We define:

\begin{align*}
\lk(e)&=\begin{array}{l}\lk(P_v^{(v)} ,P_w^{(w)})\end{array} \tag{\textit{Linking number} of $e$}\\
&\\
\E(P_v,e)&= 
\begin{array}{ll}
\lk(v+\infty,P_w) &\hspace*{0.6cm} \text{ if $P_v$ is even,}\\ 1 &\hspace*{0.6cm}\text{ otherwise.}
\end{array}
\tag{\textit{Even index} of $P_v$ wrt $e$}\\
&\\
\psi(e)&= \begin{array}{ll}
(v<w) (P_v<P_w) & \text{ if both $P_v$ and $P_w$ are odd.}\\ 
1 &\text{ otherwise.}
\end{array}\tag{\textit{Odd consistency} of $e$}
\end{align*}

The \textit{local sign} associated with the edge $e$ is
$$\sigma_\text{loc}(e,L)= \lk(e) \psi(e) \E(P_v,e) \E(P_w,e).$$

Finally, we set
$$
\varepsilon_L(e) = \sigma_\text{loc}(e,L) \sigma_\text{glo}(e,L),$$ and $$
\delta_d^i (L) = \sum_{ e \text{ admissible}}
\varepsilon_L(e) \cdot L_e,$$
and extend this formula into a linear map $ \delta_d^i: \mathcal{L}_d^{i-1} \rightarrow \mathcal{L}_d^i$.

\begin{theorem}\label{thm:maincpx}
For each $d\geq 1$, the collection of spaces $\mathcal{L}_d^*$ and maps $\delta_d^*$ forms a cochain complex of $\ZZ$-modules.
\end{theorem}

\begin{proof}
Let $e$ and $e\p$ be two edges in a leaf diagram $L$, such that $e$ is admissible. Then $e\p$ is admissible in $L_e$ if and only if $e\p$ is admissible in $L$ and $e$ is admissible in $L_{e\p}$. We call such a couple \textit{bi-admissible}. To prove the theorem, it is enough to show that for any bi-admissible couple, the contribution of $e$ and $e\p$ in the computation of $\delta^2 L$ is 0.
In other words, that the product $\varepsilon_L(e) \varepsilon_L(e\p) \varepsilon_{L_{e\p}}(e) \varepsilon_{L_e}(e\p)$ is always equal to $-1$.

If $e$ and $e\p$ are bounded respectively by $v$, $w$ and  $v\p$, $w\p$, then $(e,e\p)$ is bi-admissible if and only if $e$ and $e\p$ are admissible and the leaves $v, w, v\p$ and $w\p$ represent at least $3$ different colours. We split the proof into two parts, accordingly. 

First, assume that all leaves have pairwise different colours. In this case, every contribution from $\sigma_\text{loc}$ appears twice and cancels out. So do the contributions of $\sigma_\text{glo}$ involving other polygons than those neighboring $e$ and $e\p$. The remaining contributions of $\sigma_\text{glo}$ are summarised in Table~\ref{tab:global}; $0$ stands for \enquote{even}, $1$ for \enquote{odd}.
We show only the contribution to $\varepsilon_L(e)\varepsilon_{L_{e\p}}(e) $: the contribution to $\varepsilon_L(e\p)\varepsilon_{L_e}(e\p) $ contains exactly the opposite boolean expressions. So the point is that in each row, there is an odd number of booleans.

\begin{table}[h!]
\centering
\begin{tabular}{|c|c|c|c|c|}
\hline
\multicolumn{4}{|c|}{Parities $\stackrel{\phantom{.}}{\text{of}}$  the polygons} & \multirow{2}{*}{Total contribution of $\sigma_\text{glo}$ to $\varepsilon_L(e)\varepsilon_{L_{e\p}}(e) $}\\
 \cline{1-4} $P_v$  & $\stackrel{\phantom{.}}{P}_w$  & $P_{v\p}$ & $  P_{w\p}$&\\

\hline

$\hspace*{0.2cm}\stackrel{\phantom{.}}{0}\hspace*{0.2cm}$ &$\hspace*{0.2cm}0\hspace*{0.2cm}$ &$\hspace*{0.2cm}0\hspace*{0.2cm}$ &$\hspace*{0.2cm}0\hspace*{0.2cm}$ &$(P_{vw}<P_{v\p w\p}) $\\ \hline
$\stackrel{\phantom{.}}{0}$ &$0$ &$0$ &$1$ &$(P_{vw}<P_{w\p})$\\ \hline
$\stackrel{\phantom{.}}{0}$ &$0$ &$1$ &$1$ &$(P_{vw}<P_{v\p})(P_{vw}<P_{w\p})(P_{vw}<P_{v\p w\p})$\\ \hline
$\stackrel{\phantom{.}}{0}$ &$1$ &$0$ &$1$ &$(P_{w}<P_{w\p})$\\ \hline
$\stackrel{\phantom{.}}{0}$ &$1$ &$1$ &$1$ &$(P_{w}<P_{v\p})(P_{w}<P_{w\p})(P_{w}<P_{v\p w\p}) $\\ \hline
$\stackrel{\phantom{.}} {\multirow{3}{*}{1}}$ &$\multirow{3}{*}{1}$ &$\multirow{3}{*}{1}$ &$\multirow{3}{*}{1}$ &$(P_{v}<P_{v\p})(P_{w}<P_{v\p})(P_{vw}<P_{v\p})$\\
&&&&$(P_{v}<P_{w\p})(P_{w}<P_{w\p}) (P_{vw}<P_{w\p})$\\ &&&&$(P_{v}<P_{v\p w\p})(P_{w}<P_{v\p w\p}) (P_{vw}<P_{v\p w\p})$\\ \hline

\end{tabular}
\caption{Computation of $\delta^2$ when $\sharp \left\lbrace P_v, P_w, P_{v\p}, P_{w\p} \right\rbrace=4 $. By symmetry, there are only six cases to consider. Note that the minus sign due to $P_{v}$ and $P_{w}$ being both odd in the last line (Remark~\ref{rem:minus}) appears twice and cancels out.}\label{tab:global} \end{table}

We now assume that $v, w, v\p$ and $w\p$ represent $3$ colours, and without loss of generality that $w$ and  $w\p$ share the same one. We need not discuss the special case when $w$ is actually equal to $w\p$, since the following computations do not depend on that.
Table~\ref{tab:local} details the contribution of each factor to the product $\varepsilon_L(e) \varepsilon_L(e\p) \varepsilon_{L_{e\p}}(e) \varepsilon_{L_e}(e\p)$. The proof that the product of all contributions is always $-1$ is straightforward, using the bilinearity of $\lk$ and the formula
$$\forall a, b\in \RR,\quad \lk(a+\infty, b)=(a<b).$$

\begin{table}[h!]
\centering
\hspace*{-1cm}
\begin{tabular}{|c|c|c|c|c|c|c|}
\hline
\multicolumn{3}{|c|}{ $\stackrel{\phantom{.}}{\text{Parities}}$} & \multicolumn{4}{c|}{Contributions}\\
 \hline $P_v$  & $P_w$  & $P_{v\p}$ & $\stackrel{\phantom{.}}{\phantom{E}}\lk\stackrel{\phantom{.}}{\phantom{E}}$ & $\E$ & $\psi$ & $\sigma_{\text{glo}}$ \\
\hline

\multirow{2}{*}{${0}$} &\multirow{2}{*}{$0$} &\multirow{2}{*}{$0$} &\multirow{2}{*}{$\lk(P_v+P_{v\p},w+w\p)$} &$ \lk(P_v+P_{v\p}\!\! \stackrel{\phantom{|}}{,}w+w\p) $&\multirow{2}{*}{$1$}&\multirow{2}{*}{$1$}\\
&&&
&$\lk(v+\infty,w\p) \lk(v\p+\infty,w)$&&\\ \hline 


\multirow{2}{*}{${0}$} &\multirow{2}{*}{$0$} &\multirow{2}{*}{$1$} &\multirow{2}{*}{$\lk(P_{v\p}+v\p,w+w\p)$} &\multirow{2}{*}{$\lk(P_{v\p},w+w\p) \lk(w+\infty, w\p)$}&\multirow{2}{*}{$(v\p<w\p)(P_{v\p}<P_{vw})$}& \multirow{2}{*}{$(P_{v\p}<P_{vw})$}\\
&&&
&&&\\ \hline 


\multirow{2}{*}{$1$} &\multirow{2}{*}{$0$} &\multirow{2}{*}{$1$} &\multirow{2}{*}{$1$}&\multirow{2}{*}{$1$}&\multirow{2}{*}{$1$}&\multirow{2}{*}{$-1$}\\
&&&
&&&\\ \hline


\multirow{2}{*}{$0$} &\multirow{2}{*}{$1$} &\multirow{2}{*}{$0$} &\multirow{2}{*}{$\lk(P_v+P_{v\p},w+w\p)$} &$\lk(P_v+P_{v\p}\!\stackrel{\phantom{|}}{,}w+w\p)$&\multirow{2}{*}{$\stackrel{\phantom{.}}{1}$}&\multirow{2}{*}{$\stackrel{\phantom{.}}{1}$}\\
&&&
&$\lk(v+\infty,w\p) \lk(v\p+\infty,w)$&&\\ \hline 


\multirow{2}{*}{${0}$} &\multirow{2}{*}{$1$} &\multirow{2}{*}{$1$} &\multirow{2}{*}{$\lk(P_{v\p}+v\p,w+w\p)$}& \multirow{2}{*}{$\lk(P_{v\p},w+w\p)\lk(v+\infty, v\p)$}&\multirow{2}{*}{$(v\p<w\p)(P_{v\p}<P_{w})$}&\multirow{2}{*}{$(P_{v\p}<P_{w})$}\\
&&&
&&&\\ \hline 


\multirow{2}{*}{$1$} &\multirow{2}{*}{$1$} &\multirow{2}{*}{$1$} &\multirow{2}{*}{$1$}&\multirow{2}{*}{$1$}&\multicolumn{2}{c|}{\multirow{2}{*}{$\pm
\begin{array}{c}
 \stackrel{\phantom{|}}{\phantom{,}}(P_{v}<P_{w})(P_{v\p}<P_{w})\stackrel{\phantom{|}}{\phantom{,}}\\ (P_{v}<P_{v\p w})(P_{v\p}<P_{v w})
\end{array}$}}\\
&&&
&&\multicolumn{2}{c|}{}\\ \hline 

\end{tabular}
\caption{Contribution of each factor after obvious simplifications, in the case $P_w=P_{w\p}$. By symmetry, there are only six cases to consider. In the sixth line, $\psi$ and $\sigma_{\text{glo}}$  have the same contribution up to sign, and the sign is $+$ for $\psi$ and $-$ for $\sigma_{\text{glo}}$.} \label{tab:local} 
\end{table}

\end{proof}

\section{Simple singularities in the space of knot diagrams}\label{sec:strata}

\subsection{Germs and the associated strata}

By the \textit{space of long knots} $\K$ we refer to the (arbitrarily high, but finite)-dimensional affine approximation of the space of all smooth maps $\RR\rightarrow \RR^3$ with prescribed asymptotical behaviour, as defined in \cite{Vassiliev1990}. The \textit{discriminant} $\Sigma$ is the subset of all maps in $\K$ that are not embeddings.
A projection $p:\RR^3\rightarrow \RR^2$ endows $\K\setminus \Sigma$ with a stratification, whose strata are defined by certain semi-algebraic varieties in multijet spaces (see Example~\ref{ex:Reid}, and see \cite{David, Wall, Fiedler1parameter} and references therein for an introduction to stratified spaces and the simplest examples used in knot theory).
Those strata can be represented by Gauss diagrams with additional information of geometrical nature, i.e. involving derivatives (see \cite{Vassiliev}). We will call such a stratum \textit{simple} if the only geometric data are the writhes of the crossings -- and  \textit{geometric} otherwise.

\begin{definition}\label{def:germs}
An \textit{abstract germ} is the datum of a finite number of complete oriented graphs, together with an embedding of the disjoint union of their vertices into $\RR=\SS^1\setminus \left\lbrace \infty \right\rbrace $, such that
\begin{enumerate}
\item each graph has at least two vertices,

\item no graph has oriented cycles,

\item each edge of each graph is decorated with a sign $+$ or $-$.
\end{enumerate}

An abstract germ $\gamma$ has an underlying leaf diagram $L(\gamma)$, from which it inherits the terminology of polygons, leaves, colours, edges, as well as the Gauss and cohomological degrees $\deg$ and $\i$.
\textit{The edges of the graphs in $\gamma$ are called (signed) arrows}.
The $\ZZ$-module freely generated by germs with cohomological degree $i$ is denoted by $\mathcal{G}_i$ -- because we will essentially think of meridians, for which $i$ is the dimension.

Condition $2$ above implies that a germ induces a total order on each of its polygons, and a partial order on the set of all of its leaves, denoted by $<_\gamma$. 
A knot is said to \textit{respect} $\gamma$, or called a \textit{$\gamma$-knot}, if it maps any two leaves with the same colour to a classical crossing with over/under datum given by the order $<_\gamma$, and writhe given by the sign of the arrow between those leaves. These conditions may be inconsistent, so that no knot can respect $\gamma$; otherwise, $\gamma$ is called a \textit{topological germ}, or more simply a \textit{germ}. In that case the diagram of a generic $\gamma$-knot is uniquely determined near each imposed crossing up to local diagram isotopy.
Out of the $2^{\binom{n}{2}}$ ways to put signs on a complete graph (consistently oriented) with $n$ leaves, exactly $2^{n-1}(n-1)!$ are topological.

\end{definition}

If the leaves of $\gamma$ are fixed, the set of all $\gamma$-knots in $\K$ is an affine subspace of codimension $2\deg (\gamma)$, because there are $(\sharp P-1)$ affine equations for each polygon $P$ (which are independent if $\dim \K$ is large enough), and because the writhe conditions are open, hence $0$-codimensional. If the leaves are set free, i.e. the germ is regarded up to positive homeomorphism of the real line, then the codimension drops to $2\deg (\gamma) - \text{(number of leaves)}$, which is equal to $\i (\gamma)$.

\begin{definition}
The $\i (\gamma)$-codimensional subspace of all knots in $\K$ that respect $\gamma$ up to positive homeomorphism of the real line is denoted by  $\K_\gamma$ and called the \textit{simple stratum} associated with $\gamma$.
\end{definition}

In Subsection~\ref{sec:meridian}, we will show that a germ $\gamma$ defines canonically a co-orientation of $\K_\gamma$ (that is, an orientation of its meridian ball). That is the reason for calling it a \textit{germ}.

\begin{example}\label{ex:Reid}
The strata of codimension $1$ are described by the classical Reidemeister moves. R-I and R-II strata are geometric, and R-III is simple.
Indeed, choose a basis $(e_1, e_2, e_3)$ of $\RR^3$, such that $e_3$ is the axis of the projection $p$. This splits a knot parametrisation $f:\RR\rightarrow \RR^3$ into three coordinate functions $f_{1,2,3}$. Reidemeister strata are then defined by writhe data together with the conditions (for example):

$$\begin{array}{cclcl}
\text{\emph{R-I}} &  \exists\, x, &\begin{array}{l}f^\prime_1(x) = f^\prime_2(x) = 0\end{array}.&&\\ &&&&\\
\text{\emph{R-II}} &  \exists\, x< y,& 
\begin{array}{l}
f_1(x)=f_1(y)\\ f_2(x)=f_2(y)\\ f_3(x)<f_3(y)\end{array}
 &\text{ \emph{and} }&
\det(\begin{array}{cc}f^\prime_1(x) & f^\prime_1(y)\\
f^\prime_2(x) & f^\prime_2(y) \end{array})=0.\\ &&&&\\
\text{\emph{R-III}} &  \exists\, y< x< z,& 
\begin{array}{l}
f_1(x)=f_1(y)=f_1(z)\\ f_2(x)=f_2(y)=f_2(z)\\ f_3(x)<f_3(y)<f_3(z)\end{array}
.&&
\end{array}$$
\end{example}

Note that the conditions do not depend on the choice of a basis for the projection plane, $\left( e_1, e_2\right)$; this is a general observation, the stratification depends only on $p$.
Also, this stratification should not be confused with that of $\Sigma$ used by Vassiliev \cite{Vassiliev1990} to define finite-type cohomology classes. That one will not be used in the present paper.

\begin{remark}
When a germ is regarded up to homeomorphism, it may happen that a knot respects it in several different ways. Note however that a generic $\gamma$-knot cannot have more singularities than imposed by $\gamma$, so that the only source of multiplicity lies in two-leaved polygons, which give $0$-codimensional constraints. Rather than the strata $\K_\gamma$, one may consider  \textit{simplicial chains}, whose local weight near a given $\gamma$-knot is equal to the number of ways that knot respects $\gamma$ -- this is the implicit choice in Vassiliev's calculus \cite{Vassiliev}. Here, algebraic intersection with such chains will be modelled by means of the map $\I$ which is defined in Subsection~\ref{sec:I}.
\end{remark}

\subsection{Boundary of simple strata}

The boundary of a stratum $\K_\gamma$ is defined by the generic ways for its constraints to degenerate. There are essentially six basic ways, from which all others can be built. They can be interpreted by thinking of a generic $\gamma$-knot as a knot diagram some of whose crossings, \textit{including all multiple crossings}, are coloured in red.
\vspace{0.2cm}

\textbf{Type $\Sigma$.} Two leaves of $\gamma$ that are consecutive with respect to the order $<_\gamma$ tend to be mapped to the same point in $\RR^3$. The corresponding piece of boundary lies in $\Sigma$, so it is harmless for our purposes (understanding the cohomology of $\K\setminus \Sigma$, which is the relative homology of $\left( \K, \Sigma\right) $).
\vspace{0.2cm}

\textbf{Type $1$}. One edge of $\gamma$ whose boundary points have the same colour collapses into a point $x$, accompanied with the condition $f^\prime_1(x) = f^\prime_2(x) = 0$.
\vspace{0.2cm}

\textbf{Type \du}. Two branches of a red crossing tend to have the same direction in the knot diagram; from the point of view of $\gamma$, it results in a writhe not being well-defined any more, and replaced with either a condition of positive, or negative, collinearity of derivatives.
\vspace{0.2cm}

\textbf{Type \dd}. 
Two edges of $\gamma$ that bound a bigon in the knot diagram collapse simultaneously. This produces the same geometric condition as in Type \du.
\vspace*{0.2cm}

\textbf{Type \td}. One edge whose boundary points have distinct colours collapses.
\vspace{0.2cm}

\textbf{Type \tt}. Three edges that bound a triangle in the knot diagram collapse simultaneously.
\vspace{0.2cm}

Types $1$ to \tt~correspond to generalised Reidemeister moves, in that the crossings are allowed to be multiple. They are sorted according to how many red crossings they involve.

Besides these basic types, it can happen that types \dd, \td~and \tt~are accompanied with the simultaneous collapsing of an arbitrarily large number of triangles of type \tt. Indeed,  in all of these cases, one can see on the knot diagram that a number of crossings are locally present although they may not be imposed by the germ (red). Now these crossings may also actually be present in the germ, in which case they can either be regarded as far (which yields a basic type as above) or close, in which case they participate in the collapsing. Then, these extra crossings may be themselves multiple crossings from the beginning, and the phenomenon repeat itself.

We are now ready to define precisely which kind of degeneracies will be of interest in this paper.
\begin{definition}\label{def:ext}
We call \emph{Type $0$} a degeneracy of basic type \td~together with finitely many \emph{non-multiple} extra crossings as above -- in other words, at most two polygons with more than two leaves can be involved in the collapsing. If the two polygons of the underlying type \td~degeneracy have respectively $m$ and $n$ leaves, then there may be at most $(m-1)(n-1)$ extra arrows.
Degeneracies of basic type \td~with extra \textit{multiple} crossings are considered to fall down into Type \tt.
\end{definition}


\subsubsection*{Reidemeister farness}

We now define a class of germs that will allow us to avoid bad geometric strata and the above Type \tt~frenzy.

\begin{definition}\label{def:far}
Let $\gamma$ be a germ.
We say that two leaves in $\gamma$ are \textit{neighbours} if they are the two boundary points of an edge. Then $\gamma$ is called:
\begin{enumerate}
\item \textit{RI-close} if it contains an arrow $(v,w)$ such that $v$ and $w$ are neighbours.
\item \textit{RII-close} if it contains four distinct leaves $v, w, x, y$ such that:
\begin{itemize}
\item $v$ and $w$ are neighbours, and so are $x$ and $y$;
\item $v$ and $w$ have distinct colours;
\item $v<_\gamma x$ and $w<_\gamma y$. 
\end{itemize}
\item \textit{RIII-close} if it contains six distinct leaves $v, w, x, y, z, t$ such that:
\begin{itemize}
\item $\left\lbrace v,w\right\rbrace$, $\left\lbrace x,y\right\rbrace$, $\left\lbrace z,t\right\rbrace$ are couples of neighbours;
\item $v$, $w$ and $y$ have pairwise distinct colours.
\item $v<_\gamma x$, $y<_\gamma z$ and $w<_\gamma t$. 
\end{itemize}

\end{enumerate}
We define \textit{R-farness} of germs, and therefore of simple strata, as the negation of all of these properties. 
In other words, $\gamma$ is R-far if no generic $\gamma$-knot can be subject to a generalised Reidemeister move involving only red crossings, that is, Basic types $1$, \dd~and \tt~cannot occur.
\end{definition}
%
%

\subsection{Meridian systems and the $\partial_i$ map}\label{sec:meridian}

Roughly speaking, our goal is to define cohomology classes in the space of knots as intersection forms with R-far simple strata. This requires to understand in which meridian spheres these strata occur. By the previous discussion we mainly need to consider the meridians of simple strata. The geometric strata resulting from \du~degeneracies will later prove to be completely harmless (see Lemma~\ref{lem:du}).

Let $f$ be a knot respecting an $i$-germ $\gamma$, and $\D^i_f$ a piecewise linear (PL) $i$-disc about $f$ in $\K$, transverse to the stratification. Then the boundary of $\D^i_f$ intersects finitely many $(i-1)$-strata, at points $f_1, \ldots, f_p$, and can be covered with PL discs $\D^{i-1}_{f_k}$ with pairwise disjoint interiors. Since $\gamma$ is simple, every meridian stratum $\gamma_k$ is necessarily simple, and the degeneracy $\gamma_k \rightsquigarrow \gamma$ is either of Type $0$~(Definition~\ref{def:ext}) or \tt. \begin{definition}[reduced meridian system]
The cellular boundary map (over $\ZZ$) associated with the above decomposition of $\partial \D^i_f$ depends only on $\gamma$. It is called the \textit{meridian system} of $\gamma$. The \textit{reduced meridian system} of $\gamma$ is the induced map with target restricted to strata coming from Type $0$~degeneracies. We denote it shortly by $$\partial_\gamma: C_\gamma \rightarrow  \bigoplus_0 C_{\gamma_k}.$$
\end{definition}
When $i=0$, $\D^0_f$ consists of a single point and has a canonical orientation, i.e. there is a canonical generator of $C_\gamma\cong \ZZ$, which we denote by $1_\gamma$.

We now show that the signs $\varepsilon$ used to construct the cochain complex from Section~\ref{sec:leaf} provide a combinatorial realisation of this boundary map, and a preferred generator for each module $C_\gamma$.

\begin{definition}[$k$-splittings]
Let $\gamma$ be a germ and $\Gamma$ a graph of $\gamma$ with $n$ leaves, $n\geq 3$. A \textit{splitting of $\gamma$ along $\Gamma$} is a germ $\gamma_s$ together with the datum of a Type $0$~degeneracy $\gamma_s \rightsquigarrow \gamma$ resulting in the creation of the graph $\Gamma$. It has to involve two graphs $\Gamma_1$ and $\Gamma_2$ with respectively $k$ and $n+1-k$ leaves (we assume $k\leq n+1-k$), together with $(k-1)(n-k)$ two-leaved graphs.
If $k\geq3$, $\gamma_s$ has a favourite edge $e(s)$ which is the only edge bounded by $\Gamma_1$ and $\Gamma_2$ that gets shrunk in the degeneracy.

When $k=2$, the choice of $\Gamma_1$ (and also $\Gamma_2$ if $n=3$) and therefore $e(s)$, is not unique. However, $k$ is uniquely defined and we have a notion of \textit{$k$-splitting}.
\end{definition}

\begin{definition}\label{def:cons}
Let $\Gamma$ be a graph with two leaves in a germ $\gamma$. We define the $consistency$ $\chi(\Gamma)$ to be $+1$ if the order defined by $\RR$ and that defined by $<_\gamma$ agree on $\Gamma$, and $-1$ otherwise. We let $\chi w(\Gamma)$ denote the product of $\chi(\Gamma)$ with the sign of the arrow bounded by $\Gamma$ in $\gamma$.
\textit{The maps $\chi$ and $w$ are set to $+1$ for graphs with more than two leaves.}
\end{definition}

\begin{lemma} \label{lem:epsi}
Let $\gamma_s$ be a ($2$-)splitting of $\gamma$, and let $L(s)$ be the underlying leaf diagram to $\gamma_s$. Then the sign
$$\chi w(\Gamma_1)\chi w(\Gamma_2)\varepsilon_{L(s)}(e(s))$$
does not depend on the choice of $\Gamma_1$ and $\Gamma_2$.
\end{lemma}

This lemma is the key ingredient to show that our signs $\varepsilon$ are of topological nature. It will be proved at the end of this section.


We set:
$$\partial_i(\gamma)= \sum_{\text{all splittings}} \chi w(\Gamma_1)\chi w(\Gamma_2)\varepsilon_{L(s)}(e(s))\cdot\gamma_s $$
and extend this into a linear map $\partial_i: \mathcal{G}_i \rightarrow \mathcal{G}_{i-1}$. The reason for this map not to be graded lies essentially in the bunch of two-leaved polygons that appear as a result of splitting a germ.
\newpage
\begin{theorem}\begin{enumerate}\label{thm:top}
\item The maps $\partial_i$ and spaces $\mathcal{G}_i$ together form a chain complex.
\item There is a unique collection of maps $$\phi_\gamma: C_\gamma \hookrightarrow \mathcal{G}_{ \i(\gamma)},$$ such that $\phi_\gamma(1_\gamma) = \gamma$ if $\i(\gamma)=0$, and such that all the following diagrams commute:

\begin{center}
\begin{tikzpicture}[baseline=(current bounding box.center)]
\node (.) at (1,-0.9) {.};
\node (1b) at (-2.5,0.9) {$\mathcal{G}_{\i(\gamma)}$};

\node (2) at (0,0.9) {$\mathcal{G}_{ \i(\gamma)-1}$};

\node (3b) at (-2.5,-0.9) {$C_\gamma$};
\node (4) at (0,-0.9) {$\oplus_0 C_{\gamma_k}$};

\draw[->,>=latex] (1b) -- (2) node[midway,above] {$\partial _{\i(\gamma)}$};

\draw[->,>=latex] (3b) -- (4) node[midway,below] {$\partial_\gamma$};

\draw[right hook->,>=latex] (3b) -- (1b) node[midway,left] {$\phi_\gamma$};
\draw[right hook->,>=latex] (4) -- (2) node[midway,right] {$\oplus_0 \phi_{\gamma_k}$};


\end{tikzpicture}
\end{center}

\item The map $\phi_\gamma$ maps $C_\gamma$ isomorphically onto the submodule $\ZZ\gamma\subset \G_{\i(\gamma)}$. Hence the preimage $\phi_\gamma^{-1}(\gamma)$ defines a canonical co-orientation of the simple stratum $\K_\gamma$.
\end{enumerate}
\end{theorem}

\begin{proof}
The proof of the first point is identical with that of Theorem~\ref{thm:maincpx}. One only has to notice that the signs $\chi w$ always appear twice and cancel themselves in the computation of $\partial\circ \partial$, and that the collection of two-leaved polygons that result from a splitting does not affect the computations, because they are even polygons.

We prove points $2.$ and $3.$ simultaneously, by induction on $i$. 

When $i=0$, there is nothing to prove. The case $i=1$ also needs being treated separately. Here it suffices to notice that on the two sides of a Reidemeister III move, the sign $\chi w(\Gamma_1)\chi w(\Gamma_2)\varepsilon_{L(s)}(e(s))$ takes opposite values -- indeed, such a move switches the signs $\lk(P_1, P_2)$, $\E(P_1, e)$ and $\E(P_2, e)$, and leaves the remaining signs unchanged. So $\phi_\gamma$ is defined uniquely by mapping to $\gamma$ the generator of $C_\gamma$ that is oriented from the negative side to the positive side. Point $3.$ is then satisfied, and it implies that the direct sum of any collection of maps $\phi_\gamma$ is injective.

Now let $i\geq 2$ and assume that $2.$ and $3.$ hold up to $i-1$. The crucial point is the following.

\begin{lemma} \label{lem:connect}
If $i\geq2$, then in the cell decomposition of a meridian sphere $\SS^{i-1}_\gamma$ made of meridian discs, the union of all $(i-1)$-discs corresponding to Type $0$~degeneracies is connected.
\end{lemma}

Assuming this lemma, consider a germ $\gamma\in \mathcal{G}_i$. By definition, $\partial_i \gamma$ lies in $\oplus_0 C_{\gamma_k}$, so by induction (Point $3.$) it has a unique preimage $x$ by $\oplus_0 \phi_{\gamma_k}$. By Point $1.$ of the theorem and by induction (Point $2.$), $x$ lies in the kernel of $\oplus_0 \partial_{\gamma_k}$. In other words, it is a relative cycle in $(\SS_\gamma, \SS_\gamma\setminus \bigcup_0\D_{\gamma_k})$.
Also, it has local weight $\pm 1$, so it is a generator of $H_{i-1}(\SS_\gamma, \SS_\gamma\setminus \bigcup_0 \D_{\gamma_k})$, which by Lemma~\ref{lem:connect} is canonically isomorphic to $H_{i-1}(\SS_\gamma)\cong H_i(\D_\gamma, \SS_\gamma)\cong C_\gamma$. By pushing $x$ through these isomorphisms, we obtain a generator of $C_\gamma$, and $\phi_\gamma$ is uniquely defined by the fact that it must map this generator to $\gamma$. This terminates the proof up to Lemma~\ref{lem:connect}.

\begin{proof}[Proof of Lemma~\ref{lem:connect}]

If $\gamma$ has at least two graphs $\Gamma$ and $\Gamma^\prime$ with more than two leaves, then any two splittings respectively along $\Gamma$ and $\Gamma^\prime$ have a piece of boundary in common. If gamma has only one graph with $n>2$ leaves, then $n$ must be at least $4$ so that $i=n-2\geq 2$. Here, any two $2$-splittings sliding different branches away have a common piece of boundary, and any $k$-splitting ($k\geq 3$) has a common boundary piece with $n-1$ distinct $2$-splittings.
\end{proof}
\end{proof}


\begin{proof}[Proof of Lemma~\ref{lem:epsi}]
We first prove the result in one particular case, then proceed by induction, using a number of \enquote{moves} that allow one to join any splitting of any germ.

First note that by symmetry of the formula in $\left\lbrace \Gamma_1, \Gamma_2 \right\rbrace$ we need not check separately the case $n=3$, even though $\Gamma_2$ is not uniquely determined. Figure \ref{pic:gamman} shows a splitting $\gamma^+(2,n-1)$ of the germ with only one graph, with $n$ leaves, and only $+$ signs. The orientations of the arrows in the $(n-1)$-gon depend on the way to connect virtually the branches of the $(n-1)$-crossing; they are not shown because the sign $\varepsilon$ only depends on the underlying polygon. 
One easily sees that in $\gamma^+(2,n-1)$, $\chi w(\Gamma_1)$ is $-1$ for any choice of $\Gamma_1$, and only the maps $\lk$ and $\E$ can contribute non-trivially in $\varepsilon$.

\begin{figure}[h!]
\centering 
\psfig{file=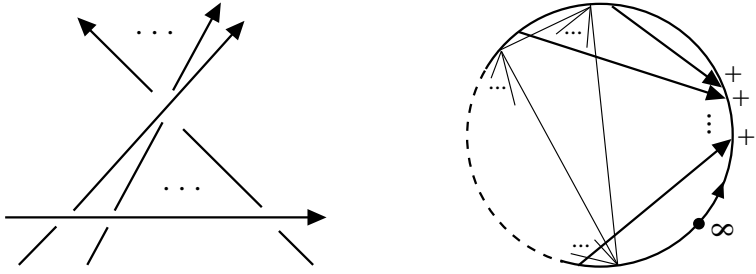, scale=0.9}
\caption{The germ $\gamma^+(2,n-1)$,  splitting of the positive $n$-branch crossing.}
\label{pic:gamman}
\end{figure}

\begin{itemize}
\item $\lk(P_1,P_2)$ is $+1$ if $\Gamma_1$ is the topmost arrow ($\Gamma_1$-candidate) in the diagram on the right of Fig.\ref{pic:gamman}, and alternates up to $(-1)^n$ for the bottom arrow.
\item $\E (P_1,e(s))$ has the same alternating property and is $-1$ for the bottom arrow
\item If $n-1$ is even, $\E (P_2,e(s))$ has the same value $+1$ for any choice of $\Gamma_1$ (and this also holds obviously if $n-1$ is odd).
\end{itemize}
This proves that the sign $\chi w(\Gamma_1)\chi w(\Gamma_2)\varepsilon_{L(s)}(e(s))$ is $(-1)^n$ for any choice of $\Gamma_1$ in $\gamma^+(2,n-1)$.

We now prove the invariance of the result under the following moves:

\begin{enumerate}
\item Adding a bystander graph;
\item Making one crossing change in  the $(n-1)$-crossing;
\item Making one crossing change at one of the $(n-1)$ $\Gamma_1$-candidates;
\item Reversing the orientation of a branch of the $(n-1)$-crossing;
\item Sliding the branch that was split away from the $n$-crossing to the other side of the $(n-1)$-crossing;
\item Changing the order in which the $n$ local branches are virtually connected;
\item Moving the point $\infty$ to another region.
\end{enumerate}

We always neglect the orientation and sign changes on $\Gamma_2$, which are harmless.
Move $1$ may only modify the contribution of $\sigma_{\text{glo}}$, but it does so in the same way for all choices of $\Gamma_1$, essentially because $\Gamma_1$ is always even. Move $2$ has no effect at all. Move $3$ only changes $\chi(\Gamma_1)$ and $w(\Gamma_1)$ into their opposite, so that $\chi w(\Gamma_1)$ remains the same. Move $6$ commutes with the other moves, so it suffices to see that it does not affect the result for $\gamma^+(2,n-1)$.

The effect of Move $5$ on the germ is identical with changing the sign of all $\Gamma_1$-candidates, which does not affect the result, and then formally applying the effect of $(n-1)$ moves of type $4$. So we are left with the two moves $4$ and $7$, shown on Fig.\ref{pic:moves47}. Move $4$ does not affect any signs for other choices of $\Gamma_1$ than the one on the picture; for this one, $\chi w$ is changed into its opposite, and so is $\E (P_1, e(s))$. If $n-1$ is odd, nothing else changes; otherwise, both the linking number of $e(s)$ and the even index of $P_2$ are also reversed.

Move $7$ changes $\E (P_1, e(s))$ into its opposite for all choices of $\Gamma_1$ other than the one we can see on Fig.\ref{pic:moves47} -- and it has no other effect for them. For the last choice of $\Gamma_1$, it changes $\chi$, and nothing else if $n-1$ is odd; otherwise it also changes both even indices of $P_1$ and $P_2$.
\begin{figure}[h!]
\centering 
\psfig{file=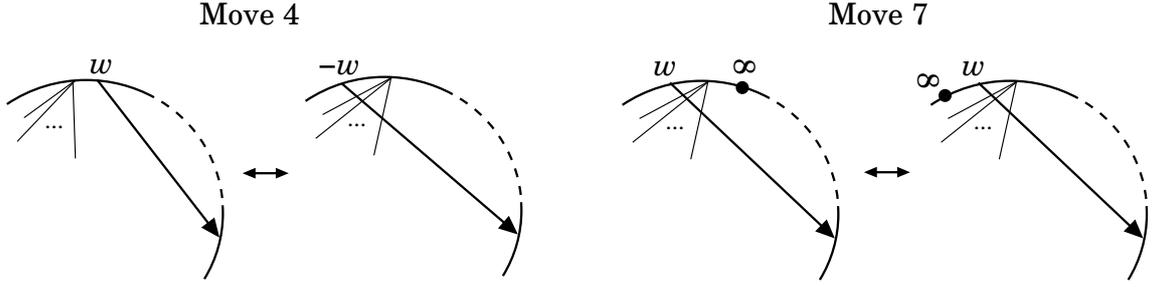, scale=0.9}
\caption{The effect of moves $4$ and $7$ on germs.}
\label{pic:moves47}
\end{figure}
\end{proof}

\section{Main result}\label{sec:main}

We now introduce a degenerate version of arrow diagrams, designed to \textit{count subgerms} in the spirit of \cite{PolyakViro}. Subgerms are the algebraic artifact that allows one to see whether a knot respects a germ, and in how many ways. They are also natural in that if $\gamma \rightsquigarrow \gamma^\prime$ is a degeneracy, then $\gamma$ may appear in the meridian of $\gamma^\prime$ only as a subgerm.

\subsection{Tree diagrams}

Let $P$ be a polygon in $\SS^1\setminus \left\lbrace\infty \right\rbrace$ of cardinality greater than $1$. A \textit{spanning tree} for $P$ is a collection of ordered couples $(v,w)$ with $v,w\in P$, still called \textit{arrows}, such that the corresponding abstract oriented graph is a tree. The number of arrows in a spanning tree is always equal to the cardinality of the underlying polygon minus $1$.

A \textit{tree diagram} is a finite collection of pairwise disjoint polygons in $\SS^1\setminus \left\lbrace\infty \right\rbrace $ endowed with spanning trees. We keep denoting such diagrams by the letter \enquote{$A$} to respect the tradition of arrow diagrams, and save \enquote{$T$} for single spanning trees. Tree diagrams naturally inherit the Gauss and cohomological degrees defined for leaf diagrams, namely:
\begin{itemize}
\item The Gauss degree $\deg(A)$ of a tree diagram is equal to its total number of arrows.
\item The cohomological degree $\i(A)$ is the Gauss degree minus the number of colours (trees).\end{itemize}

Again, tree diagrams are regarded up to positive homeomorphisms of the real line $\SS^1\setminus \left\lbrace \infty\right\rbrace $. The $\ZZ$-module freely generated by equivalence classes of tree diagrams of degree $d$ and codimension $i$ is denoted by $\mathcal{A}_d^i$. Note that $\mathcal{A}_d^i$ is trivial whenever $i$ is greater than $d-1$, and whenever $i$ or $d$ is negative (see Remark~\ref{rem:isolated}).

\subsubsection*{The triangle relation}

Observe that a spanning tree $T$ defines a partial order on the underlying polygon: say that $v<_T w$ if $T$ contains the arrow $(v,w)$, and extend this definition by transitivity -- which is possible because $T$ is a tree. We say that $T$ is \textit{monotonic} if the relation $<_T$ is total. Accordingly, a tree diagram is called \textit{monotonic} if all of its trees are so. Monotonic spanning trees for a given polygon $P$ are in one-to-one correspondence with total orders on $P$. Denote by $\nabla(T)$ the set of all monotonic spanning trees that correspond to total orders compatible with $<_T$.

\begin{definition}\label{def:tri} The \textit{triangle relation} is the equivalence relation on $\mathcal{A}_d^i$ generated by the equalities
\begin{equation}\label{eq:tri}
A=\sum_{T\p \in \nabla(T)} A_{T\p},
\end{equation}
where $A$ is a tree diagram that contains $T$ as a spanning tree and $A_{T\p}$ is the diagram obtained from $A$ by replacing $T$ with $T\p$. We denote the quotient $\ZZ$-module by $\tilde{\mathcal{A}}_d^i$. It is naturally isomorphic with the subspace of $\mathcal{A}_d^i$ spanned by monotonic tree diagrams.\end{definition}

\begin{remark}
This relation originated in the work of M.Polyak on arrow diagrams -- see \cite{P1}, and also \cite{PolyakTalk, PVCasson}. It owes its name to the fact that it is locally generated by the relation schematically depicted on Fig.\ref{pic:tri}.\end{remark}
\begin{figure}[h!]
\centering 
\psfig{file=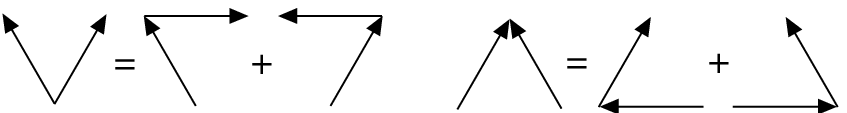, scale=1.2}
\caption{Local triangle relations -- only a part of a spanning tree is shown; the remaining invisible parts should be identical for all three diagrams in a given equality.} \label{pic:tri}
\end{figure}

\subsubsection*{Reidemeister farness}

\begin{definition}
The \textit{Reidemeister farness} of monotonic diagrams is defined similarly to that of germs (Definition~\ref{def:far}).The submodule of $\tilde{\mathcal{A}}_d^i$ generated by R-far monotonic diagrams is denoted by
$\tilde{\mathcal{A}}_{d, \text{far}}^i$. This definition makes sense since any $\alpha\in \tilde{\mathcal{A}}_d^i$ has a unique representative involving only monotonic diagrams.
\end{definition}

\subsection{The pairing of tree diagrams with germs}\label{sec:I}

\begin{definition}[partial germs and signed tree diagrams]
A \textit{partial germ} is a leaf diagram whose every polygon is enhanced into a \textit{connected} abstract graph with oriented and signed arrows. The difference with germs is that here the graphs need not be complete. A partial germ whose every graph is a tree is called a \textit{signed tree diagram}. 

Partial germs inherit the degrees $\deg$ and $\i$ from their underlying leaf diagrams. The corresponding $\ZZ$-modules of signed tree diagrams are denoted by $\mathcal{T}_{i,d}$ and $\mathcal{T}_i$.
\end{definition}

\begin{definition}
A \textit{subgerm} of a germ $\gamma$ is the result of forgetting an arbitrary number of its arrows, in such a way that every graph corresponding to a polygon with more than two leaves remains connected -- although two-leaved polygons may completely disappear. This condition means that subgerms must remember the codimension of $\gamma$ -- \textit{but the Gauss degree may drop}.

We set $\I(\gamma)$ to be the formal sum of all subgerms of $\gamma$ that are signed tree diagrams. It is understood that subgerms are counted with multiplicity if the removal of distinct sets of arrows yields homeomorphic results. This defines a linear map 
$$\I: \mathcal{G}_i \rightarrow \mathcal{T}_i.$$
If $\tau$ is a signed tree diagram, we define $\T(\tau)$ as the underlying tree diagram, multiplied by the \textit{product of the signs of all arrows of two-leaved polygons}. Again this extends into a linear map 
$$\T:\mathcal{T}_{i,d} \rightarrow \mathcal{A}_d^i.$$

\begin{remark}
The fact that the map $\T$ disregards the signs of arrows associated with polygons that have more than two leaves should be interpreted this way: for these polygons, the signs of the arrows have already contributed by entering the co-orientation defined by germs on their associated strata. In other words, when a simple crossing merges with others into a multiple crossing, we stop regarding its writhe as making sense individually. See Lemma~\ref{lem:epsi} and Theorem~\ref{thm:top}.
\end{remark}
\end{definition}

\begin{definition}
For $\alpha\in \tilde{\mathcal{A}} _d^i$
and $\gamma\in \mathcal{G}_i$, we set  $$\<\!\<\alpha, \gamma\>\!\>= \left\langle \alpha, \T\circ\I (\gamma) \right\rangle,$$
where $\left\langle \cdot,\cdot \right\rangle$ is the Kronecker delta on tree diagrams, extended by bilinearity.
\end{definition}

We have to prove that this is a good definition, that is:

\begin{lemma}
Let $\nabla\in \mathcal{A}_d^i$ be a triangle relator, i.e. the difference between the two sides of Equation~\eqref{eq:tri}.
Then: $$\forall \gamma\in \mathcal{G}_i,\hspace*{0.2cm}
\<\!\<\nabla, \gamma\>\!\>=0.$$
\end{lemma}
\begin{proof}
We may assume that $\gamma$ is a single germ. The result follows then from the facts that the graphs of $\gamma$ are complete and consistently oriented, and that the map $\T$ disregards the signs of crossings in trees with more than two leaves.
\end{proof}

This elementary proof should be compared with that of  \cite[Lemma $1.9$]{FT1cocycles}. There, the result was deeply related with the fact that the germ was topological. Here, all the topology is confined in the co-orientation associated with germs, and this lemma actually holds for abstract germs.

\begin{definition}
Let $c$ be a PL $i$-chain in $\K\setminus \Sigma$ that is transverse to the stratification. Then $c$ intersects finitely many simple $i$-strata $\gamma_p$, with intersection numbers $\eta_p$ defined by the co-orientation from Theorem~\ref{thm:top}.
For $\alpha\in \tilde{\mathcal{A}} _d^i$, we set: $$\left\lbrace \alpha,c \right\rbrace = \<\hspace*{-0.1cm}\< \alpha, \sum_p \eta_p \gamma_p \>\hspace*{-0.1cm}\> .$$\end{definition}

We are now in a position to see why degeneracies of type \du~do not deserve particular attention.

\begin{lemma}\label{lem:du}
Let $\zeta$ be an \emph{almost simple stratum}, that is, a boundary component of a simple $i$-stratum corresponding to a Type \du~degeneracy.
Let $\SS_\zeta$ be the meridian sphere of $\zeta$. Then:
$$\forall d\geq 0,\forall \alpha\in \tilde{\mathcal{A}} _{d,\text{far}}^i, \left\lbrace \alpha,\SS_\zeta \right\rbrace = 0.$$
\end{lemma}
It means that the cocyclicity condition for R-far cochains is empty around such strata.

\begin{proof}
Denote by $a$ the arrow in the germ $\gamma$ that is subject to \du~degeneracy. The situation is quite different according to whether or not $a$ is part of a multiple crossing.

First assume that $a$ is isolated. Then $\SS_\zeta$ intersects exactly two simple strata $\gamma_0$ and $\gamma_\pm$, corresponding respectively to $\gamma$ with the arrow $a$ forgotten, and $\gamma$ with the arrow $a$ duplicated into two arrows with opposite writhe, that intersect or not depending on the geometric condition of $\zeta$. We have exactly the two sides of a usual Reidemeister II move. Moreover, since the co-orientation of a germ depends only on the configuration of its graphs with more than two leaves, $\gamma_0$ and $\gamma_\pm$ induce opposite orientations on $\SS_\zeta$, so that, up to sign:
$$\left\lbrace \alpha,\SS_\zeta \right\rbrace = \<\hspace*{-0.06cm}\< \alpha, \gamma_0 - \gamma_\pm \>\hspace*{-0.06cm}\> .$$
The result follows by classical arguments.
\begin{figure}[h!]
\centering 
\psfig{file=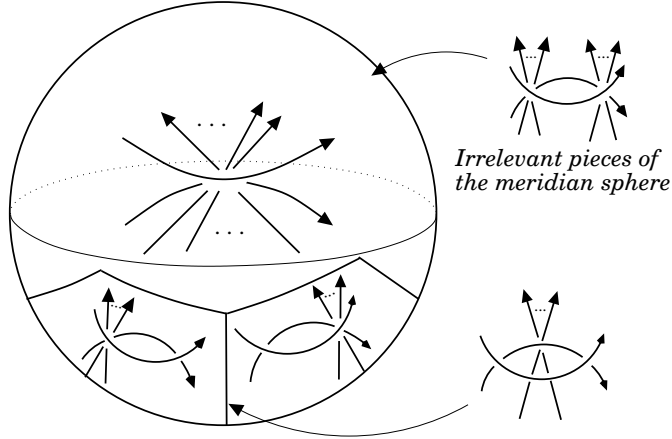, scale=0.70}
\caption{Meridian of an almost simple stratum (case of a multiple crossing).} \label{pic:du}
\end{figure}

Now assume that $a$ is part of a multiple crossing, with $k\geq 3$ branches (two of which have tangent projections). This time $\SS_\zeta$ intersects $2^{k-2}$ simple $i$-strata, obtained from $\zeta$ by duplicating $a$ into two arrows with opposite sign, and then form two new multiple crossings by sharing the remaining $k-2$ branches among those two. However, one of the two arrows $a_+$ and $a_-$ must remain isolated so that subdiagrams stand a chance to be R-II far.
Hence only two diagrams may contribute, $\gamma_+$ and $\gamma_-$, as indicated by Fig.\ref{pic:du}. One can see on the picture that they have a piece of boundary in common (in fact, two): that is the key allowing us to compare their orientations. Indeed, $\gamma_+$ and $\gamma_-$ induce the same orientation on their common boundary, hence they induce opposite orientations on $\SS_\zeta$, and again, up to sign:
$$\left\lbrace \alpha,\SS_\zeta \right\rbrace = \<\hspace*{-0.06cm}\< \alpha, \gamma_+ - \gamma_- \>\hspace*{-0.06cm}\> .$$
Now since $\alpha$ is R-II far, the isolated duplicate of $a$ must be deleted for a subdiagram to contribute, so that the relevant subdiagrams in $\gamma_+$ are also subdiagrams in $\gamma_-$, with only difference given by the sign of $a$. But this sign is disregarded by $\<\hspace*{-0.06cm}\< \cdot, \cdot \>\hspace*{-0.06cm}\>$, because $a$ is a part of a multiple crossing.
\end{proof}

\subsection{Cohomology of tree diagrams and of the space of knots}

Given a tree diagram $A$, an edge is called admissible if it is so in the underlying leaf diagram $L$. For such an edge $e$ there is a natural way to define a tree diagram $A_e$ which is a lift of $L_e$. Namely, if $e$ is bounded by the leaves $v$ and $w$, the arrows of $A_e$ are the arrows of $A$ where $w$ is replaced with $v$ every time it appears. This edge-shrinking process is compatible with the triangle relations. 
We define a linear map $ \tilde{\delta}_d^i: \tilde{\mathcal{A}}_d^{i-1} \rightarrow \tilde{\mathcal{A}}_d^i$
on the generators by
$$
\tilde{\delta}_d^i (A) =  \sum_{ e \text{ admissible}}
\chi(\Gamma_v)\chi(\Gamma_w)\varepsilon_L(e) \cdot L_e,$$
with the consistency $\chi$ as in Definition~\ref{def:cons}.

We are now ready for the main theorem of this paper.

\begin{theorem}\label{thm:main}
\begin{enumerate}
\item The collection of maps $\tilde{\delta}_d^i$ and sets $\tilde{\mathcal{A}}_d^{i}$ forms a graded, finite cochain complex. We denote by $H_{d,\text{far}}^i$ the submodule of those $i$-th homology classes in degree $d$  that have a representative cocycle in $\tilde{\mathcal{A}} _{d,\text{far}}^i$.

\item (Stokes formula) For any $d\geq 0$, $i\geq 1$, $\alpha\in \tilde{\mathcal{A}} _{d,\text{far}}^{i-1}$
and $\gamma\in \mathcal{G}_i$, 
$$\<\hspace*{-0.08cm}\<\tilde{\delta}_d^i (\alpha), \gamma\>\hspace*{-0.08cm}\>= \<\hspace*{-0.06cm}\<\alpha, \partial_i (\gamma)\>\hspace*{-0.06cm}\>.$$
\item There is a natural map
$$H_{d,\text{far}}^i \rightarrow H^i(\K\setminus \Sigma)$$
induced by the pairing $\left\lbrace \cdot,\cdot \right\rbrace  $.
For $i=0$, the image of this map consists of invariants induced by homogeneous Goussarov-Polyak-Viro formulas for long virtual knots \cite{GPV}. For $i=1$, the image consists of arrow-germ formulas as defined in \cite{FT1cocycles}.
\end{enumerate}
\end{theorem}

\begin{remark}
The farness constraint could be lightened, by allowing R-III close diagrams. In the case $i=0$, this is harmless (there are no additional equations) thanks to  \cite[Lemma~$3.2$]{MortierPolyakEquations}, and it yields all GPV invariants \cite[Theorem~$3.6$]{MortierPolyakEquations}. For higher values of $i$, it would require to compute the proper $\varepsilon$ signs to associate with Type \tt~degeneracies, and to consider subgerms whose graphs are not necessarily trees. 

One could also think of removing the R-I,II farness condition -- by contrast, this would require to handle arbitrary geometric strata, resulting in a far more complicated story. For $i=0$ it is pointless, R-I,II farness is actually a necessary condition for cocyclicity \cite[Lemma~$3.4$]{MortierPolyakEquations}. For $i=1$ it brings no new cohomology classes \cite[Theorem~$2.11$]{FT1cocycles}.
\end{remark}

\begin{conjecture}\label{conj}
The image of the map $H_{d,\text{far}}^i \rightarrow H^i(\K\setminus \Sigma)$ consists of Vassiliev cohomology classes of degree at most $d$.
\end{conjecture}

This is known to hold for $i=0$, and to hold over $\ZZ_2$ when $d=i+1$ (extreme case of diagrams with only one tree), and when $i=1$ and $d=3$ (case of the Teiblum-Turchin cocycle \cite{Turchin, Vassiliev}). 

\begin{proof}[Proof of Theorem~\ref{thm:main}]
$1.$ This follows from Theorem~\ref{thm:maincpx} after noticing that the additional contribution $\chi(\Gamma_v) \chi(\Gamma_w)$ always cancels itself out in $\tilde{\delta}\circ\tilde{\delta}$.\\

$2.$ For simplicity we omit the indices and exponents in the maps $\partial_i$ and $\tilde{\delta}_d^i$. We also may assume that $\alpha$ is a tree diagram and $\gamma$ a germ. 
Note that $\alpha$ cannot be a subdiagram of both a $k$-splitting and an $l$-splitting of $\gamma$ for $k\neq l$; so the proof can be split according to the at most unique value of $k$ such that the RHS stands a chance to be non-zero when $\partial (\gamma)$ is restricted to $k$-splittings. As a last preliminary, note that we prove the formula at the level of $\mathcal{A}_{d,\text{far}}^i$, i.e. \textit{before the quotient by triangle relations}. 

If $k>2$, then because $\alpha$ is R-far we see that any subdiagram of a term in $\partial (\gamma)$ that contributes non-trivially to the RHS must have gotten rid of every two-leaved graph that resulted from the splitting. Similarly, if $k=2$, at most one of these graphs may have survived. Also, if no one of them has survived, then the subdiagram's possible contribution is cancelled out by the corresponding subdiagram in the \emph{opposite} $2$-splitting (where the sliding branch has been pushed in the opposite direction).
Thus we see that for any value of $k$, we can restrict $\I(\partial (\gamma))$ to certain subdiagrams such that the corresponding subdiagrams of $\gamma$ are \textit{signed tree diagrams}.

We now use a divide and conquer trick. Note that the subdiagrams to which we restricted $\I(\partial (\gamma))$ have a well-defined preferred edge $e(s)$. So we can
arrange the non trivial contributions to the RHS according to which edge of $A$ corresponds to $e(s)$. This edge must clearly be admissible in $A$, so a corresponding arrangement can be realised in the LHS. 
Now it is easy to see that the contributions in each pack are naturally in $1$-$1$ correspondence, and that the signs match.\\

$3.$ The map $\alpha\mapsto \left\lbrace \alpha,\cdot \right\rbrace$ makes tree diagrams into cochains in $\K$. By Theorem~\ref{thm:top}, Lemma~\ref{lem:du} and the Stokes formula, it maps cocycles to cocycles and coboundaries to coboundaries, thus inducing a map $H_{d,\text{far}}^i \rightarrow H^i(\K\setminus \Sigma)$.

For $i=0$, the map $ \tilde{\delta}_d^1$ is isomorphic with the map $d^\Lambda$ from \cite{FT1cocycles} restricted to Gauss degree $d$, and this isomorphism is compatible with the Stokes formulas. There, it is proved that Goussarov-Polyak-Viro invariants are exactly the kernel of a certain map $d^{\Lambda} \oplus d^{\Delta}\oplus d^{\I} \oplus d^{\I \!\I}$, and our R-farness condition ensures that the diagrams live in the kernel of  $d^{\Delta}\oplus d^{\I} \oplus d^{\I \!\I}$.

For $i=1$, we use the result and terminology of \cite[Theorem~$2.11$]{FT1cocycles}. By our R-farness condition the condition of the theorem is satisfied, and also the \textit{cube equations} associated with \includegraphics [scale=0.6]{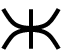}-strata are empty. Now it is straightforward to check that the tetrahedron equations associated with \includegraphics [scale=0.6]{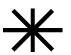}-strata yield the kernel of the map $ \tilde{\delta}_d^2$
restricted to edges that are bounded by one leaf from the triangle, and the remaining equations from \includegraphics [scale=0.6]{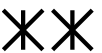}-strata are encoded by the restriction of $ \tilde{\delta}_d^2$ to the complementary set of edges.
Finally, considering the number of leaves in the polygons, the kernel of $ \tilde{\delta}_d^2$ is the intersection of the kernels of these two restrictions.
\end{proof}

\section{Examples and comments}

An essential aspect of our construction is that it is of a virtual nature. That is, the equations do not care about the fact that the germs at which we evaluate the bracket $\<\hspace*{-0.08cm} \<\alpha,\cdot \> \hspace*{-0.08cm}\> $ may or may not correspond to classical knots.
A major benefit is that it makes the theory simple and computable.
Taking care of classicalness would be much more complicated: to the best of our knowledge there is no complete characterisation of Gauss diagrams of classical knots that do not require to actually try drawing the knot -- although there are some powerful invariants allowing to detect non-classicalness in a lot of cases, using for instance the Gaussian parity \cite{ManturovParity}, the Miyazawa polynomial \cite{Naoko} or the arrow polynomial \citep{DyeKaufArrowPoly}. 

On the side of drawbacks, the map $H_{d,\text{far}}^i \rightarrow H^i_\text{Vassiliev}(\K\setminus \Sigma)$, assuming that Conjecture~\ref{conj} holds, 
is unlikely to be surjective. For instance, the Vassiliev invariant of order $3$ given by the  Gauss diagram formula $v_3$ in \cite[Theorem $2$]{GPV} cannot be found in $H_{3,\text{far}}^0$ (a virtual version of $v_3$ is constructed in \cite{PolyakChmutovHOMFLYPT}, but its non-homogeneity makes it of a strongly different nature). However, our cochain complex produces a formula for $v_3$, quite unexpectedly, not from $H_{3,\text{far}}^0$ but from $H_{3,\text{far}}^1$, by integrating a $1$-cocycle over the Fox-Hatcher loop. The following is a corollary of Theorem~\ref{thm:int}.

\begin{theorem}
The tree-diagram formula $\tilde{\alpha}_3^1$ on Fig.\ref{pic:a31} is an R-far $1$-cocycle. Moreover, the integration of $\tilde{\alpha}_3^1$ on the Gramain loop and the Fox-Hatcher loop of a knot $K$ yield respectively the Gauss diagram formulas:
$$\begin{array}{ccccl}
\int_{\rot(K)} \tilde{\alpha}_3^1  &=&
\<\hspace*{-0.1cm}\<  \raisebox{-0.39cm}{\includegraphics [scale=1]{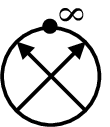}},K \>\hspace*{-0.1cm}\>  &=&
v_2(K),\\&&&&\\
\int_{\FH(K)} \tilde{\alpha}_3^1 &=&\<\hspace*{-0.1cm}\<  \raisebox{-0.35cm}{\includegraphics [scale=1]{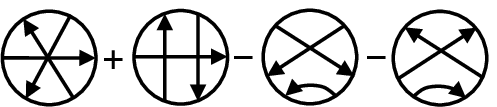}},K \>\hspace*{-0.1cm}\>  &=&
6v_3(K)-w(K)v_2(K),
\end{array}$$
where $w(K)$ denotes the blackboard framing of the diagram of $K$ considered. In particular the map $H_{3,\text{far}}^1 \rightarrow H^1(\K\setminus \Sigma)$ has rank at least $1$. If Conjecture~\ref{conj} holds, then this rank is $1$ and $\tilde{\alpha}_3^1$ is a realisation of the Teiblum-Turchin cocycle over the integers.
\end{theorem}

As far as we know, this is the first time a Gauss diagram formula specific to classical knots is found without using Gauss diagram identities (see \cite{Ostlund}). The only step where we did leave the comfortable field of virtual arguments is when we used the existence of the Fox-Hatcher loop!

Note that the second formula is unbased -- which is a general phenomenon when integrating over the Fox-Hatcher loop. The evaluation bracket is then defined similarly to the based version, but it counts subdiagrams with multiplicity, that is given by the order of their symmetry group (see \cite[Sections~$2.2$ and $2.4$]{Ostlund}, and \cite[Section~$4.1.2$]{VKTG}).

\begin{figure}[h!]
\centering 
\psfig{file=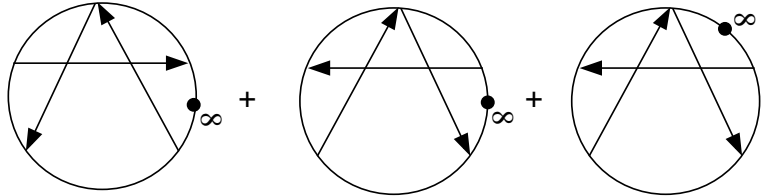, scale=0.65}
\caption{The $1$-cocycle $\tilde{\alpha}_3^1$.} \label{pic:a31}
\end{figure}

\subsection{Formal integration of $1$-cocycles}

A deep result due to Hatcher~\cite{HatcherTopologicalmoduli} states that the connected component of $\K\setminus\Sigma$ corresponding to a non-satellite long knot $K$ has the homotopy type of $\SS^1$ if $K$ is a torus knot, and of $\SS^1\times \SS^1$ if $K$ is hyperbolic. For those knots there are essentially two interesting elements in $H_1(\K_K)$: the \textit{Gramain loop} $\rot(K)$, which consists of a rotation of a long knot around its axis, and the \textit{Fox rolling}, or \textit{Hatcher loop} $\FH(K)$, which consists of sliding the ball at infinity (in $\SS^3$) along the knot.

The Gramain loop does not depend on the Reidemeister moves we use to represent it. However, the Fox-Hatcher loop depends on a framing choice: indeed, each time one adds $+1$ to the framing of $K$, the ball at infinity makes one positive full spin on itself, which amounts to a negative spin of $K$, hence it adds $-\rot(K)$ to $\FH(K)$.

\begin{figure}[h!]
\centering 
\psfig{file=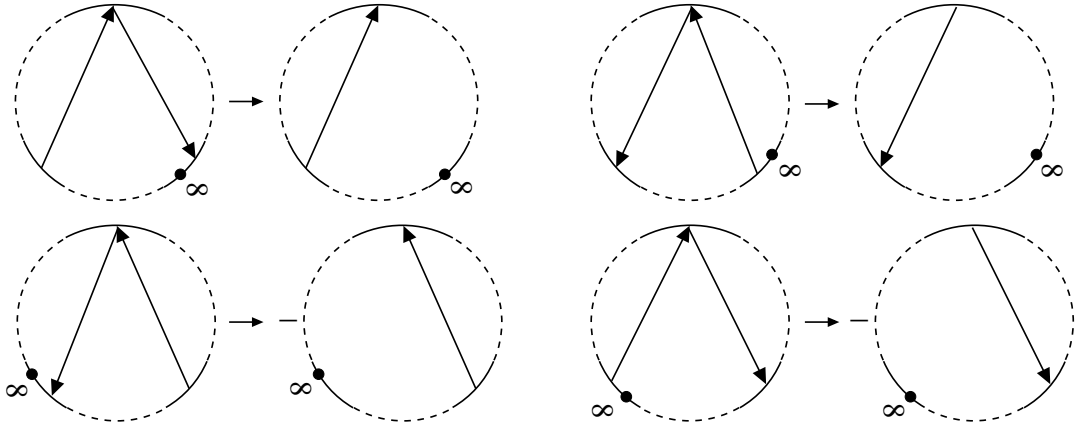, scale=0.65}
\caption{Sign rules for $\int_{\rot}^h$ (on the left) and $\int_{\rot}^l$ (on the right).} \label{pic:introt}
\end{figure}
\begin{figure}[h!]
\centering 
\psfig{file=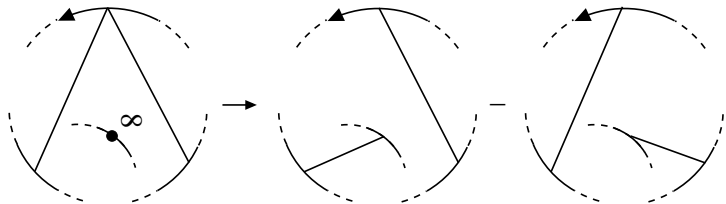, scale=0.65}
\caption{Sign rule for $\int_{\FH}$. It does not depend on the position of the point $\infty$.} \label{pic:inthat}
\end{figure}

Let $A$ be a monotonic tree diagram of codimension $1$ -- so that its only polygon with more than $2$ leaves is a triangle $T$. If the highest (resp. lowest) point of this triangle with respect to the order $<_A$ is also the lowest (resp. highest) of all leaves in $A$ with respect to the $\RR$ order, then we define a new diagram $\int_{\rot}^h A$ (resp. $\int_{\rot}^l A$) of codimension $0$, by forgetting the arrow containing that point, with sign rule as indicated on Fig.\ref{pic:introt}. Otherwise, we set $\int_{\rot}^h A=0$ (resp. $\int_{\rot}^l A=0$). This defines linear  maps $\tilde{\mathcal{A}} _d^1 \rightarrow \tilde{\mathcal{A}} _{d-1}^0$.

With the same notations, let $(a,b)$ and $(b,c)$ denote the two arrows of $T$. We construct two unbased diagrams by replacing the arrow $(a,b)$ with $(a,\infty)$ (resp. $(b,c)$ with $(\infty,c)$) while \textit{forgetting the point $\infty$}, and give them signs depending only on the relative position of $a$, $b$ and $c$ in the cyclic order -- see the rule on Fig.\ref{pic:inthat}. The difference is denoted by $\int_{\FH} A$ and defines a map $\tilde{\mathcal{A}} _d^1 \rightarrow \tilde{\mathcal{A}} _d^0$.

\begin{theorem}\label{thm:int}
Let $\alpha\in \tilde{\mathcal{A}}_ {d,\far }^1 \cap \Ker \delta_d^2$. Then for any classical knot $K$:
\begin{enumerate}
\item
$$
\begin{array}{ccc}
\int_{\rot(K)}\alpha &=& \<\hspace*{-0.1cm}\<  \int_{\rot}^h \alpha + \int_{\rot}^l \alpha ,K \>\hspace*{-0.1cm}\>. \end{array}
$$ 
In particular, the right-hand side defines a finite-type invariant of $K$ of degree at most $d-1$. However, $\int_{\rot}^h \alpha + \int_{\rot}^l \alpha$ might not lie in $\Ker \delta_{d-1}^1$.
\item $$
\begin{array}{ccc}
\int_{\FH(K)}\alpha &=& \<\hspace*{-0.08cm}\<  \int_{\FH} \alpha ,K \>\hspace*{-0.08cm}\>. \end{array}
$$ 
The right-hand side defines a \emph{regular} invariant of $K$. Its value on a diagram of $K$ with trivial blackboard framing defines a finite-type invariant of $K$ of degree at most $d$. [Recall that here the bracket on the right counts subdiagrams with their potential multiplicity due to symmetry.]
\end{enumerate}
\end{theorem}

This theorem can be proved by analysing the presentation of $\rot$ from \cite[Fig.$144$]{FiedlerQuantum1cocycles}, and that of $\FH$ given by Fox \cite{FoxRolling} from the viewpoint of Gauss diagrams -- as in the proof of \cite[Theorem~$3.3$]{FT1cocycles}. Reidemeister farness is crucial in the proof, not only for the theory to work properly, but to have a good control of the non-trivial contributions to the integrals. For example, the $1$-cocycle formula from \cite[Theorem $3.2$]{FT1cocycles}, which allows R-III close diagrams, is impossible [to us] to integrate directly on the Fox-Hatcher loop, because of uncontrollable contributions.

\subsubsection*{Gauss diagram identities}

This theorem can be useful even when applied to a cocycle that is trivial in $H^1(\K\setminus \Sigma)$. Indeed, it may happen that \textit{the integration of such a cocycle is not formally zero}. When this happens, it means that we have found a Gauss diagram identity, that is, a formula for the trivial invariant. But since there are no such formulas for virtual knots, we have there a non-trivial obstruction to classicalness.

Among the low-degree examples, we have thereby a new proof that the Gauss diagram formulas
$$
\<\hspace*{-0.08cm}\<   \raisebox{-0.39cm}{\includegraphics[scale=1]{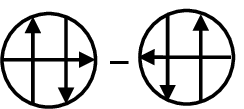}}\,,K \>\hspace*{-0.08cm}\> 
\text{ and }
\<\hspace*{-0.08cm}\<   \raisebox{-0.6cm}{\includegraphics[scale=1]{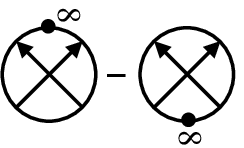}}\,,K \>\hspace*{-0.08cm}\> 
$$
vanish for classical knots.

\subsection{Higher degree examples}

A number of higher degree formulas comes for free as in general  $\tilde{\mathcal{A}}_{i+1,\text{far}}^i \cong H_{i+1,\text{far}}^i$, whose rank grows at least quadratically with $i$. All of those can be proved to be Vassiliev classes at least over $\ZZ_2$ using the homological calculus from \cite{Vassiliev}. One could study their non-triviality by using the results of \cite{BudneyHomotopyType, BudneyHomology} which are an excellent sequel, state of the art and completion to Hatcher's work on the topology of spaces of knots. We study here the cocycles in $H_{3,\text{far}}^2$.

Our main motivation for computing higher degree examples lies in reinterpreting a result of Budney et al. \cite{BudneyConantScanellSinha}, which states that it is possible to compute the invariant $v_2$ by counting an appropriate kind of quadrisecants with appropriate signs.

In the present language, a quadrisecant of a knot is a particular direction of projection for which the knot respects a germ with one polygon and four leaves. Hence, counting quadrisecants with signs is precisely what $2$-cocycles in $H_{3,\text{far}}^2$ do. More precisely, given a knot $K$,
consider a sphere in $\RR^3$, centered at the origin and with radius large enough to intersect $K$ only in two points where it is arbitrarily close to its axis.
Each point in that sphere defines a different direction of projection,  except for the two intersection points with the axis of $K$.  So we do not have a $2$-cycle, but still a canonical $2$-chain, where evaluating our cocycles makes sense since generically the quadrisecants stay far away from the knot axis during an isotopy of $K$. We call that $2$-chain $\SS_\infty(K)$.

The module $\tilde{\mathcal{A}}_{3,\text{far}}^2$ has two generators $v_3^2=$ \raisebox{-0.39cm}{\includegraphics [scale=1]{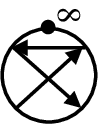}} and $\tilde{v}_3^2=$ \raisebox{-0.39cm}{\includegraphics [scale=1]{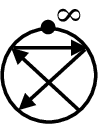}}, and both are cocycles. 

\begin{theorem}
For any knot $K$,
$$v_2(K)=\<\hspace*{-0.08cm}\<   v_3^2,\SS_\infty(K) \>\hspace*{-0.08cm}\> 
=\sum_{\includegraphics[scale=1]{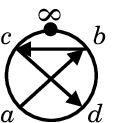}} w(a,b)w(c,d),
$$
where the sum is over all quadrisecants of $K$ of the indicated type, and $w(a,b)$ denotes the writhe of the simple crossing between the branches $a$ and $b$.
\end{theorem}

One can see that this is a new point of view on \cite[Proposition~$6.2$]{BudneyConantScanellSinha}, with a much simpler formula to think of. Indeed, the quadrisecants counted by $v_3^2$ are precisely those which \enquote{determine the cycle $(1342)$} in the language of \cite [Section~$6$]{BudneyConantScanellSinha}.

\begin{proof}
We begin with the second equality. It is proved by analysing the co-orientation defined by $v_3^2$ and understanding what orientation it defines on $\SS_\infty(K)$. The natural co-orientation of the plane on Fig.\ref{pic:coor}, as defined by Theorem~\ref{thm:top}, is counterclockwise if and only if the product of writhes $pqr$ is $+1$.

\begin{figure}[h!]
\centering 
\psfig{file=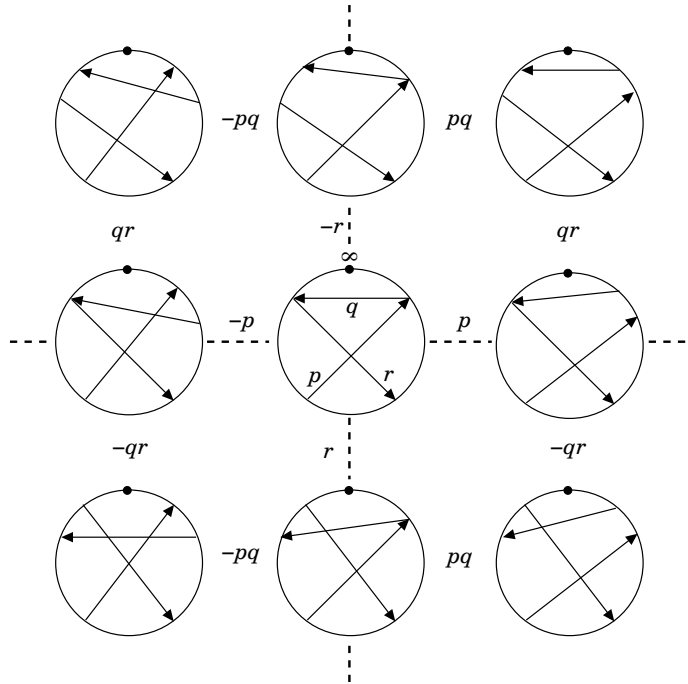, scale=1}
\caption{The meridian of a germ with underlying tree diagram $v_3^2$. The numbers $p$, $q$, $r$ are the writhes of the arrows as indicated in the middle diagram. A sign between two diagrams indicates the co-orientation.} \label{pic:coor}
\end{figure}

 Now we need to draw the picture of Fig.\ref{pic:coor} on the sphere $\SS_\infty$. For this, observe the following. Choose a point $x$ on $\SS_\infty$ which defines a diagram $K_x$ with exactly one generic triple point, say $f(t_1)> _x f(t_2)>_x f(t_3)$; the set of such points is a $1$-submanifold $X\subset\SS _\infty$. By moving the center of $\SS_\infty$ so that it lies on the line containing the triple point, the derivatives $f^\prime(t_1)$, $f^\prime(t_2)$ and $f^\prime(t_3)$ project to a generic triple of vectors $v_1, v_2, v_3\in T_x \SS_\infty$.

\textbf{Fact.} The direction of $T_x X$ lies in the angular region determined by the directions of $v_1$ and $v_3$ that does not contain the direction of $v_2$.

To see it, think of the top and bottom branches as locally spiraling around the medium branch.

Fig.\ref{pic:sinfty} reads like this. Independently of the direction of the furthest branch ($4$), we know that the branch of $X$ that slides $4$ away and keeps the triple point $\left\lbrace 1,2,3\right\rbrace$ lies in the region bounded by $1$ and $3$ that does not contain $2$. Also, it appears that the orientation $p/-p$ (defined by the middle horizontal line of Fig.\ref{pic:coor}) depends only on the orientation of the branch $3$ as indicated.
It is then easy to see that the splitting of the remaining triple point is supported by the direction $2$, and  that the orientation $qr/-qr$ depends only on the orientation of the branch $2$.

To conclude, the relative position of $p$ and $qr$ on Fig.\ref{pic:sinfty} is dictated by the sign $q$ (writhe of the crossing between branches $2$ and $3$). Hence the orientation induced on $\SS_\infty$ by $v_3^2$ is dictated by the sign $pqr\cdot q=pr$, which is the result announced.
\begin{figure}[h!]
\centering 
\psfig{file=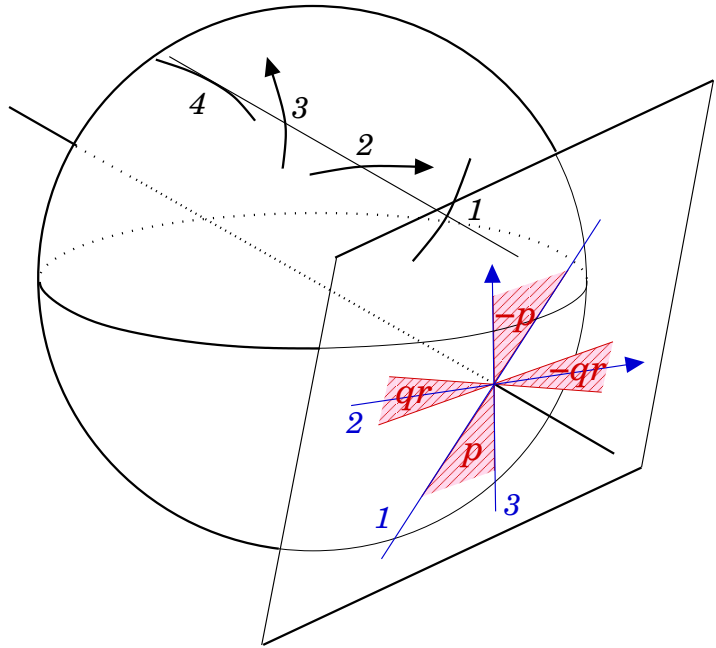, scale=0.9}
\caption{Orientation of $\SS_\infty$ induced by $v_3^2$} \label{pic:sinfty}
\end{figure}

Using this formula, it is straightforward to see that $\<\hspace*{-0.08cm}\<   v_3^2,\SS_\infty(K) \>\hspace*{-0.08cm}\> 
$ is a Vassiliev invariant of degree at most $2$; therefore it suffices to check the first equality for the trefoil.
\end{proof}

\bibliographystyle{plain}
\bibliography{bibli}

\end{document}